\documentclass[11pt]{article}
\usepackage{amsmath,amsfonts,amssymb}
\usepackage{a4wide, amsthm}
\usepackage[english]{babel}
\usepackage[babel=true]{csquotes}
\usepackage{textcomp}

\usepackage{graphicx}
\usepackage[usenames,dvipsnames]{color}
\usepackage[colorlinks=true, pdfstartview=FitV, linkcolor=black, citecolor=blue, urlcolor=blue]{hyperref}
\def\Xint#1{\mathchoice
{\XXint\displaystyle\textstyle{#1}}%
{\XXint\textstyle\scriptstyle{#1}}%
{\XXint\scriptstyle\scriptscriptstyle{#1}}%
{\XXint\scriptscriptstyle\scriptscriptstyle{#1}}%
\!\int}
\def\XXint#1#2#3{{\setbox0=\hbox{$#1{#2#3}{\int}$ }
\vcenter{\hbox{$#2#3$ }}\kern-.6\wd0}}

\def\dashint{\Xint-}
\usepackage{authblk}

\usepackage{hyperref} 
\usepackage{mathrsfs}
\renewcommand{\div}{\operatorname{div}}
{\newtheorem{thm}{Theorem}[section]}
{\newtheorem{prop}[thm]{Proposition}}
{\newtheorem{coro}[thm]{Corollary}}
{\newtheorem{lemme}[thm]{Lemma}}
{}
{\newtheorem{rem}{Remark}[section]}

\newcommand{\supp}{\text{supp }}

\newcommand{\R}{\mathbb{R}}
\newcommand{\T}{\mathbb{T}}
\newcommand{\N}{\mathbb{N}}

\newcommand{\pa}{{\partial}}
\newcommand{\na}{{\nabla}}

\newcommand{\eps}{{\varepsilon}}

\def\div{\hbox{div \!}}

\title{Analysis of a sedimenting suspension  \\ near a vertical wall }
\author[1]{David Gérard-Varet \footnote{david.gerard-varet@imj-prg.fr}}
\author[1]{Amina Mecherbet\footnote{mecherbet@imj-prg.fr}}
\affil[1]{Université Paris Cité and Sorbonne Université, CNRS, IMJ-PRG, F-75013 Paris, France.}
\begin{document}
\maketitle

\begin{abstract}
We consider a sedimenting suspension in a Stokes flow, in the presence of a vertical wall. We study the effect of a particle-depleted fluid layer near the wall on the bulk dynamics of the suspension. We show that this effect can be captured by an appropriate wall law of Navier type. We provide in this way a rigorous justification of the apparent slip observed in many experiments. We also discuss the phenomenon of intrinsic convection predicted in some physics articles. 
\end{abstract}

\section{Introduction}
The mathematical analysis of  suspensions of small particles sedimenting in a Stokes flow has been the subject of extensive research over the last years. In the context of inertialess particles, a classical model is the following. Given a finite or infinite collection of identical non-overlapping spherical particles $(B_i = B(X_i,R))_{i \in I}$, set in a bounded or unbounded domain $\Omega$, one considers the system 
\begin{equation} \label{stokes_equation}
\begin{aligned}
-\mu \Delta u + \na p & = f, \quad \text{ in } \: \Omega \setminus(\cup_i B_i) \\
\div u & = 0,  \quad \text{ in } \: \Omega \setminus(\cup_i B_i) \\
 u\vert_{\pa \Omega} &  = 0 \\
D(u) & = 0 \quad \text{ in }  B_i,  \quad \forall i \\
\int_{\pa B_i} \sigma_\mu(u,p) n & =  -m g e, \quad \int_{\pa B_i} \sigma_\mu(u,p) n \times (X-X_i)= 0, \quad  \forall i  
\end{aligned}
\end{equation}
The unknowns of this system are the velocity field and pressure:
$$u(x) =  (u_1,u_2,u_3)(x_1,x_2,x_3), \quad p(x) = p(x_1,x_2,x_3) $$
The first three lines describe a Stokes flow with viscosity $\mu$ outside the particles, with the usual no-slip condition at $\pa \Omega$. The fourth one expresses rigidity of the velocity field of each particle $B_i$, and is equivalent to 
$$ u = u_i + \omega_i \times (x-X_i) \quad \text{in } B_i$$
for some unknown translational and angular velocities $u_i, \omega_i \in \R^3$.The two relations in the last line, correspond to balances of forces and torques, and involve the Newtonian stress tensor $\sigma_\mu(u,p) = 2\mu D(u) - p I$. The constants $u_i, \omega_i$ can be seen as Lagrange multipliers associated with these constraints. The term $-mge$, $e = (0,0,1)^t$ encodes gravity and buoyancy effects: denoting $\rho_p$ and $\rho_f$ the particle and fluid masses per unit volume, one has $m = \frac{4}{3} \pi R^3 (\rho_p - \rho_f)$, while $g$ is the usual gravitational constant. This is the term responsible for sedimentation. 

Roughly, all mathematical studies of this model divide into two categories
\begin{itemize}
    \item those that focus on steady properties of the suspension. Time evolution of $X_i = X_i(t)$ is neglected (or frozen), the particle distribution is considered as given, and mean properties of the suspension, such as the average settling velocity \cite{HillairetHoefer23, Hasimoto59, Batchelor72Sedimentation, DuerinckxGloria22} or the effective viscosity  \cite{HainesMazzucato12,NiethammerSchubert19, Gerard-Varet19, Gerard-VaretHillairet19, Gerard-VaretHoefer21, Gerard-VaretMecherbet20,DuerinckxGloria19}, are deduced. The derivation of the mean settling velocity, and its expansion in terms of the solid volume fraction are notoriously difficult, and rely on subtle cancellation mechanisms which require strong assumptions on the  distribution of the particles (either periodic like or random stationary with strong mixing properties).   
\item those that analyse the dynamics of the suspension (for a large but finite $N \gg 1$ collection of balls $B_1, \dots, B_N$): the time dependence of $X_i = X_i(t)$ is taken into account through coupling of \eqref{stokes_equation} with the ODE 
\begin{equation}
\dot{X}_i = u_i 
\end{equation}
In the dilute case, under appropriate lower bound on the minimal distance between the particles, a mean field analysis allows to derive a transport-Stokes equation on $\rho = \lim_{N\rightarrow \infty} \frac{1}{N} \sum_{i=1}^N \delta_{X_i}$. We refer to \cite{Hofer18MeanField,Mecherbet19,Hofer&Schubert,HoferSchubert23} for all details. 
\end{itemize}
Except for the work \cite{HillairetHoefer23} that contains a brief discussion of boundary effects, all mathematical references above consider unbounded domains $\Omega = \R^3$ or $\Omega = \T^3$. On the other hand, the impact of the boundaries of a container on the nature of the sedimentation process has been widely discussed in the physics community.  One of the most striking examples of such impact is the so-called Boycott effect \cite{Boy20}: the settling speed  of a suspension is much faster in an inclined channel than in a vertical one, as the slurry of particles accumulating on the upward-facing wall slides faster to the bottom, creating a counterflow in the clear fluid layer above. 
Even in a vertical channel, the role of boundaries has been emphasized for decades. One main reason for this boundary effect is the formation of a so-called depletion layer near the wall, where the concentration of the suspension is much less than in the bulk of the suspension, and where large velocity gradients occur. This depletion may be due to excluded volume effect: for instance, for spherical particles, one has the constraint $d(X_i, \pa \Omega) > R$, which results in a poorer concentration in a layer of thickness $\approx R$. But additional mechanisms, such as shear induced migration, may enhance the depletion mechanism and much enlarge the layer \cite{SegSil,SegSil2,Ho_Leal_1974,Ausserre}, notably for suspensions of polymers. 
In relation to this depletion layer, two phenomena are much discussed in the literature: 
\begin{itemize}
    \item The first one, observed experimentally, is  {\em apparent slip}: for suspensions flowing through a channel, one observes non-zero downward velocities just outside the depletion layer, that can be interpreted as additional slip : see \cite{Abbasi,Koponen,Ghosh} and references therein. 
    \item The second one, not observed experimentally but predicted in several studies \cite{GeiMaz,BFAH96,BFBA98}, is called {\em intrinsic convection}:  homogeneous suspensions are expected to experience an additional downward flow in the bulk, while  experiencing  upward counterflows just outside the depletion layers (resulting in zero additional net flux). 
\end{itemize}
At first sight, these two phenomena may seem contradictory (downward {\it vs} upward flow just outside the depletion layers).  One ambition of this paper is to clarify these different behaviours. More globally, we wish here to analyze rigorously the effect of a depletion layer on a suspension, starting from the simplest model \eqref{stokes_equation}. For $\eps > 0$ we consider the half-space $\Omega = \Omega^\eps = (-\eps,+\infty) \times \R^2$, where the depletion layer $D_\eps =  (-\eps,0) \times \R^2$, of width $\eps$, is free of particles centers. Namely, 
\begin{equation} \label{eq_assumption}
X_1, \dots, X_N \in K \Subset \overline{\Omega^0}, \quad \Omega^0  = \R^*_+ \times \R^2. 
\end{equation}
Taking 
$$L = |K|^{1/3}, \quad U = \frac{mgN}{\mu L}, \quad P = \frac{\mu U}{L}$$
as reference length, velocity and pressure, we can put system \eqref{stokes_equation} in dimensionless form: 
\begin{align*}
&    X' = X/L, \quad \eps' = \eps/L, \quad R'= R/L \\
&    u' = u/U, \quad p' = p/P, \quad f' = \frac{L^2}{\mu U} f, \quad X'_i  = X_i/L, \quad B'_i = B_i/L = B(X'_i, R').  
\end{align*}
After dropping the primes, we end up with 
\begin{equation} \label{main_equation}
\begin{aligned}
-\Delta u + \na p & = f, \quad \text{ in } \: \Omega^\eps \setminus(\cup B_i) \\
\div u & = 0,  \quad \text{ in } \: \Omega  ^\eps\setminus(\cup B_i) \\
 u\vert_{\pa \Omega^\eps} &  = 0 \\
D(u) & = 0 \quad \text{ in }  B_i, \quad 1 \le i \le N \\
\int_{\pa B_i} \sigma(u,p) n & = -\frac{1}{N}e, \quad \int_{\pa B_i} \sigma(u,p) n \times (x-X_i)= 0,  \quad 1 \le i \le N  
\end{aligned}
\end{equation}
where $\sigma(u,p) = 2D(u) - p I$. Note that we could have considered some variation of \eqref{main_equation} with additional forcing $\int_{B_i} f$ and torque $\int_{B_i} f \times (x-X_i)$  included in the last two relations. 
We want to understand the asymptotic behaviour of the solution $u = u^{N,\eps}$ for large $N$ and small $\eps$. We make three main assumptions:
\begin{itemize}
    \item We denote by $\rho^N$ the empirical measure defined by 
    $
    \rho^N = \frac{1}{N} \sum_i \delta_{X_i} 
    $
and consider $\rho$ a bounded density compactly supported in $\overline{\Omega^0}$.
\item We assume  that there exists $\theta>1$ such that 
\begin{equation} \label{H1}
\epsilon> \theta R
 \end{equation}
which is a slight reinforcement of the non-overlapping constraint between the spheres and the wall. 
\item Denoting $d_{min} = \min_{i \neq j} |X_i - X_j|$,  we assume that  
\begin{equation}\label{hyp_dmin}
    d_{\min}\geq \max\left (\frac{C}{N^{1/3}},2\theta R \right)
\end{equation}
\end{itemize}
Under these assumptions, we will perform a two-step analysis of \eqref{main_equation}.
\paragraph{Step 1 : approximation by a continuous model.}
We will first show that the solution $u^{N,\eps}$ of \eqref{main_equation} can be approximated by the continuous solution $u^\eps$ of 

\begin{equation} \label{limit_equation}
\begin{aligned}
-\Delta u^\eps + \na p^\eps & = f -  \rho e, \quad \text{ in } \: \Omega^\eps  \\
\div u^\eps & = 0,  \quad \text{ in } \: \Omega^\eps \\
 u^\eps \vert_{\pa \Omega^\eps} &  = 0 
\end{aligned}
\end{equation}
We assume in all what follows that $f \in L^1(\Omega^\eps) \cap L^\infty(\Omega^\eps)$. In particular,  $f \in L^{6/5}(\Omega^\eps)$  so that \eqref{main_equation}, resp. \eqref{limit_equation}, has a unique solution $u^{N,\eps}$, resp. $u^\eps$ in
$$V = \{ u \in H^1_{loc}(\overline{\Omega^\eps}), \na u \in L^2(\Omega^\eps)\}$$
We will prove in Section \ref{sec_approx}:
\begin{thm} \label{thm1}
Let   $1<q<3/2$, $Q \Subset \overline{\Omega^\eps}$.  There exists  $C > 0$  depending on $q$, $Q$ and on the constant $\theta$ in \eqref{H1}, such that , 
\begin{equation} \label{first_estimate}
 \|u^{N,\eps} - u^\eps\|_{L^q(Q)} \le C \Big( \phi + \|\rho^N-\rho\|_{(W^{2,q'}(\Omega^0))^*}+R \Big) 
 \end{equation}
 where $\phi=RN^3$ the particles volume fraction.
\end{thm}
\begin{itemize}
\item The norm $\|\rho^N-\rho\|_{(W^{2,q'}(\Omega^0))^*}$ can be replaced by the first Wasserstein distance $W_1(\rho^N,\rho)$ between $\rho^N$ and $\rho$ and consequently by  $W_p(\rho^N,\rho)$ for any $p\in[1,+\infty]$ if one assumes $\rho \in L^1(\Omega^0)$.
    \item The control on the minimal distance \eqref{hyp_dmin} can be relaxed to $d_{\min}>2 \theta R $ together with any assumption ensuring a uniform bound on $$
\frac{1}{N} \underset{j \neq i}{\sum} \frac{1 }{|X_i-X_j|^2} \leq C,
$$
see Remark \ref{discussion_d_min} for more details.
\end{itemize}
In cases without boundaries, $\Omega = \R^3$ or $\Omega = \T^3$, this kind of statement is already known, see for instance \cite{Hofer&Schubert}. The proof is based there on an approximation of $u^{N,\eps}$ made of so-called point Stokeslets and point Stresslets, centered at the $X_i$'s. They involve the Oseen tensor, that is the fundamental solution of the Stokes operator in $\R^3$,  and its derivatives. They are not fitted to our half-space case, as the inhomogeneous Dirichlet data that they create at the boundary is too large. We must therefore adapt this strategy, constructing the analogue of point Stokeslets  for a half-space with homogeneous Dirichlet condition. We believe that this part of the analysis is of independent interest, and  contributes to the interesting problem of extending the method of reflections to domains with boundaries \cite{Hoefer19}. 

\paragraph{Step 2 : Wall law.}
Once the continuous model \eqref{limit_equation} has been derived, the goal is to understand the effect of the depletion layer $D^\eps$ on the solution $u^\eps$. In other words, one would like to improve the crude approximation  
\begin{equation} \label{crude_approx}
\begin{aligned}
-\Delta u^0 + \na p^0 &= f -  \rho e  , \quad \text{ in } \: \Omega  \\
\div u^0 & = 0,  \quad \text{ in } \: \Omega^0 \\
 u^0 \vert_{\pa \Omega^0} &  = 0 
\end{aligned}
\end{equation}
where the impact of the layer is totally neglected. In this regard, we will prove in Section \ref{sec_wall_law} the following theorem: 
\begin{thm} \label{thm2}
Assume that $f$ and $\rho$ are smooth, compactly supported in $\overline{\Omega^0}$. Then, for all $m \in \N$, one has  
$$ u^\eps = u^0 + \eps u^1 + \dots +  \eps^m u^m + O(\eps^{m+1}) \quad \text{ in } \: \dot{H}^1(\Omega^0)$$
where $u^0$ solves \eqref{crude_approx}, and for each $i \ge 1$, $u^i$ solves an homogeneous Stokes equation in $\Omega^0$ with inhomogeneous Dirichlet data at $\pa \Omega^0$ coming from lower order profiles.
\end{thm}
This theorem provides an asymptotic expansion of $u^\eps$ at arbitrary order, where each successive correction refines the description of the effect of the depletion layer.  One could relax the assumptions on the smoothness of $\rho$ and $f$: for instance, it is enough that $\na u^0,p^0$ belong to  $H^{m+2}(\Omega^0)$ to have an expansion with $m$ terms.  

The analysis leading to Theorem \ref{thm2} shows in particular that the inhomogeneous Dirichlet data for $u^1$ is given by 
\begin{equation} \label{corrector_equation}
 u^1 \vert_{\pa \Omega^0}   = (0, \pa_1 u^0_{2},\pa_1 u^0_{3})\vert_{\pa \Omega^0} 
\end{equation}
 This will imply, see Section \ref{sec_wall_law}:
 \begin{coro} \label{coro1}
Let $\rho$ and $f$ as in the previous theorem, and $u^S$ the solution of the Stokes system with Navier slip boundary condition: 
 \begin{equation} \label{improved_approximation}
\begin{aligned}
-\Delta u^S + \na p^S & = f -   \rho e  , \quad \text{ in } \: \Omega^0  \\
\div u^S & = 0,  \quad \text{ in } \: \Omega^0 \\
 u^S\vert_{\pa \Omega^0}   & =   \eps (0, \pa_1 u^S_2, \pa_1 u^S_3)\vert_{\pa \Omega^0} 
\end{aligned}
\end{equation}
Then, one has the estimate 
$$\|u^\eps-u^S\|_{\dot{H}^1(\Omega^0)} = O(\eps^2) $$
 \end{coro}
This $O(\eps^2)$ estimate shows that imposing a wall law of Navier type at the artificial boundary 
$\pa \Omega^0$ improves the $O(\eps)$ error estimate provided by the homogeneous Dirichlet condition. It is the rigorous translation of the notion of apparent slip evoked in several papers, \cite{Abbasi,Koponen, Ghosh}. We will further discuss it  in Section \ref{sec_final_comments}. The other phenomenon mentioned above, that is {\em intrinsic convection} is more subtle and corresponds to a degenerate regime of what we consider here. We will also discuss it in Section \ref{sec_final_comments}.

\section{Stokes problem on the half space: fundamental solution and Stokeslets} \label{sec_stokes}
For the analysis carried in this section \ref{sec_stokes} and the next section \ref{sec_approx}, there is no restriction in taking $\eps = 0$, that is working on $\Omega^0 =(0,+\infty)\times \R^2$. The results adapt straightforwardly to $\Omega^\epsilon=(-\varepsilon,+\infty)\times \R^2$ by  translation.

We introduce the fundamental solution to the Stokes equation on the half space $(0,+\infty)\times \R^2$. Given $y \in (0,+\infty)\times \R^2$, we set $x \mapsto G(x,y)$ the unique solution to 
\begin{equation}
\left\{
\begin{array}{rcl}
-\Delta u+\nabla q &=& \delta_{y} \mathbb{I} , \text{ on } \Omega^0 \\
\div(u)&=&0, \text{ on } \Omega^0 \\
u\vert_{x_1=0}&=&0
\end{array}
\right. 
\end{equation}
where $\mathbb{I}$ is the identity matrix in $3d$. We have for all $x \neq y$
\begin{equation}\label{ineq:fundamental_solution}
|G(x,y)| \leq \frac{C}{|x-y|}, \quad |\nabla_y G(x,y)|+ |\nabla_x G(x,y)| \leq \frac{C}{|x-y|^2}
\end{equation}
The proof of such an estimate is postponed to Section \ref{appendix}.
Let us also introduce the Stokes velocity for a sedimenting sphere in a half space. Given $F\in \R^3$, $r>0$ and $y_0 \in \Omega^0$ such that $(y_0)_1>r$,  we use the notation $x \mapsto \mathcal{U}^{st}_{r}[F](x,y_0) \in \dot{H}^1(\Omega^0)$ for the unique solution of the Stokes equation
\begin{equation} \label{eq_stokeslet}
\left\{
\begin{array}{rcl}
-\Delta u+\nabla q &=& 0,  \text{ on } \Omega^0 \setminus \overline{B(y_0,r)}\\
\div(u)&=&0,  \text{ on }\Omega^0 \setminus \overline{B(y_0,r)}\\
Du &=&0, \text { on } B(y_0,r)\\
u\vert_{x_1=0}&=&0 
\end{array}
\right.
\end{equation}
$$
\int_{\partial B(y_0,r)} \sigma(u,p) n =F,\quad 
\int_{\partial B(y_0,r)} [\sigma(u,p) n] \times(x-y_0)=0 $$
Note that $\mathcal{U}^{st}_{r}[F](\cdot,y_0)$ is implicitly given by a rigid vector field in  $B(y_0,r)$. The associated pressure will be denoted by $\mathcal{P}^{st}_{r}[F](\cdot,y_0)$. We have the following 
\begin{prop}\label{prop:stokeslet}
Assuming that $ (y_0)_1> \theta r$ for some $\theta>1$: 
\begin{enumerate}
\item We have for all $x \in \Omega^0 \setminus B(y_0,r)$ \label{ref:it1}
$$
\mathcal{U}^{st}_{r}[F](x,y_0)= \frac{1}{r} \mathcal{U}^{st}_{1}[F]\left(\frac{x}{r},\frac{y_0}{r} \right)
$$

\item For any $r>0$ and $y_0 \in \Omega^0$ satisfying $(y_0)_1 >\theta r$, there exists $\mathcal{H}_r[F](\cdot,y_0)$  such that for all $ x \not \in B(y_0,\theta r)$  \label{ref:it2}
\begin{equation}\label{eq:decomposition_stokeslet}
\mathcal{U}^{st}_{r}[F](x,y_0)= G(x,y_0)F + \mathcal{H}_r[F](x,y_0)
\end{equation}
\begin{equation}\label{eq:decay_rest_stokeslet}
\left|\mathcal{H}_r[F](x,y_0) \right|\leq r \frac{C_\theta|F| }{|x-y_0|^2}, \quad \left|\nabla_x \mathcal{H}_r[F](x,y_0) \right|\leq r \frac{C_\theta |F|}{|x-y_0|^3},
\end{equation}
\item The energy satisfies \label{ref:it3} $$\|\nabla \mathcal{U}^{st}_r[F](\cdot,y_0)\|_{L^2(\Omega^0)}^2=2\|D \mathcal{U}^{st}_r[F](\cdot,y_0)\|_{L^2(\Omega^0)}^2$$ and  
\begin{equation}\label{eq:energy}
\|\nabla \mathcal{U}^{st}_r[F](\cdot,y_0)\|_{L^2(\Omega^0)}\leq \frac{C|F|}{\sqrt{r}}
\end{equation}
with a constant $C>0$ independent of $y_0$, $F$ and $r$.
\end{enumerate}
\end{prop}
In order to prove such a result we first recall the following extension result which can be obtained by standard scaling arguments. 
\begin{lemme}\label{lemme:extension}
Let $p>1$, $\theta>1$, $a \in \R^3$, $\lambda>0$  and $v \in W^{1,p}(B(a,\lambda))$ a divergence free velocity field with vanishing mean on $B(a,\lambda)$. There exists an extension $u_v \in W^{1,p}_0(B(a,\theta \lambda))$ 
 of $v$ which is divergence free and such that 
$$
\|\na u_v\|_{L^p(B(a,\theta \lambda) \setminus B(a,\lambda))} \leq C_\theta \|\nabla v\|_{L^p(B(a,\lambda))}
$$
with a constant depending only on $\theta$ and $p$ and not on $\lambda$ and $a$.
\end{lemme}

\begin{proof}[Proof of Proposition \ref{prop:stokeslet}] 
\textit{Proof of item \ref{ref:it1}.} The proof relies on a standard scaling argument.\\
\textit{Proof of item \ref{ref:it3}.}
The first identity is valid for  any divergence-free vector field $v$ vanishing at $x_1 = 0$: it follows from the identity $-\Delta v = -2 \div(Dv)$, multiplying by $v$ and integrating by parts. As regards  the energy bound, we use the estimates: 
\begin{multline*}
\|\nabla \mathcal{U}^{st}_1[F](\cdot,y_0)\|_{L^2((0,+\infty)\times \R^2)}^2=2\|D \mathcal{U}^{st}_1[F](\cdot,y_0)\|_{L^2((0,+\infty)\times \R^2)}^2\\
=2\|D \mathcal{U}^{st}_1[F](\cdot,y_0)\|_{L^2((0,+\infty)\times \R^2 \setminus B(y_0,1))}^2= F \cdot \oint_{B(y_0,1)} \mathcal{U}^{st}_1[F](\cdot,y_0) \leq C |F| \| \nabla \mathcal{U}^{st}_1[F](\cdot,y_0) \|_{L^2}
\end{multline*}
Note that for the last inequality we have used the Sobolev imbedding
$$ \{ u \in \dot{H}^1(\Omega^0), \quad u\vert_{\pa \Omega^0}= 0 \}   \subset L^6(\Omega^0). $$
For general $r$, we then use  the scaling argument of item 1.

\textit{Proof of item \ref{ref:it2}}. We need to show that
$$\mathcal{H}_r[F](\cdot,y_0)  :=\mathcal{U}^{st}_{r}[F](\cdot,y_0) - G(\cdot,y_0) F $$
satisfies the inequalities in \eqref{eq:decay_rest_stokeslet}. Note that $\frac{1}{r}G(\frac{x}{r},\frac{y_0}{r})=G(x,y_0)$. Together with item \ref{ref:it1}, it shows that  $\mathcal{H}_r[F]$ inherits the scaling property: 
$\mathcal{H}_{r}[F](x,y_0)= \frac{1}{r} \mathcal{H}_{1}[F]\left(\frac{x}{r},\frac{y_0}{r} \right)$.
This allows to restrict to $r=1$.  We have for all $x \in \Omega^0 \setminus B(y_0,\theta)$
\begin{multline*}
\mathcal{U}^{st}_{1}[F](x,y_0) - \left(\dashint_{ B(y_0,1)} G(x,z) d \sigma(z) \right) F = \\
\int_{\partial B(y_0,1)} \left(G(x,y)-\dashint_{ B(y_0,1)} G(x,z) d z \right)\sigma(\mathcal{U}^{st}_{1}[F](y,y_0)) n d \sigma (y)
\end{multline*}
Hence, using Lemma \ref{lemme:extension}, we denote by $v\in H^1_0(B(y_0,\theta))$ a divergence free extension of $y \mapsto G(x,y)-\dashint_{ B(y_0,1)} G(x,z) d z$ such that 
$$
\| \nabla v \|_{L^2(B(y_0,\theta))} \leq C_\theta \|y\mapsto \nabla_y G(x,y) \|_{L^2(B(y_0,1))}
$$
with a constant depending only on $\theta$, one gets
\begin{multline*}
\left|\mathcal{U}^{st}_{1}[F](x,y_0)- \left(\dashint_{ B(y_0,1)} G(x,z) d z \right) F \right|= \left|\int_{B(y_0,\theta)\setminus B(y_0,1)} \nabla v : \nabla \mathcal{U}^{st}_{1}[F](\cdot,y_0)\right|\\
\leq \| \nabla \mathcal{U}^{st}_{1}[F](\cdot,y_0)\|_{L^2} \| \nabla v \|_{L^2(B(y_0,\theta))}
\end{multline*}
We use then the energy estimate for $ \mathcal{U}^{st}_{1}[F]$ of item \ref{ref:it3}
together with inequality \eqref{ineq:fundamental_solution} which yields
\begin{equation}\label{ineq1}
\| \nabla v \|_{L^2(B(y_0,\theta))}  \le C_\theta \|y\mapsto \nabla_y G(x,y)\|_{L^\infty(B(y_0,1))} \leq \frac{C'_\theta}{|x-y_0|^2} 
\end{equation}
where we used that $|x-y|>(1-1/\theta)|x-y_0|$ provided that $|x-y_0|>\theta$. Hence, we write
\begin{align*}
\mathcal{H}_1[F](\cdot,y_0) &  = \mathcal{U}^{st}_{1}[F](\cdot,y_0) - \left(\dashint_{ B(y_0,1)} G(\cdot,z) dz \right) F
+\left(\dashint_{ B(y_0,1)} [G(\cdot,z)  - G(\cdot,y_0)] dz \right) F
\end{align*}
and the extra term  $\left(\dashint_{ B(y_0,1)} [G(\cdot,z)  - G(\cdot,y_0)] dz \right) F$  can be estimated analogously using \eqref{ineq1}. This concludes the proof of the first inequality in \eqref{eq:decay_rest_stokeslet}. The second one is similar, applying $\na_x$ to all quantities above, and using the second bound in \eqref{ineq:fundamental_solution} instead of the first one.  

\end{proof}

\section{Approximation of the velocity field} \label{sec_approx}

The main goal of this section is to prove Theorem \ref{thm1}. There is no loss of generality in proving it in the case  $\eps = 0$. We use the shortcuts $u^N = u^{N,0}$, $F=-\frac{1}{N}e $. Let us denote  
$$
u_{\text{app}}= \underset{i}{\sum} \mathcal{U}^{st}_R\left[F\right](x,X_i)+ u_{f,N}
$$
where we recall that $\mathcal{U}^{st}_R[F](\cdot,X_i)$ is the Stokeslet attached to $B_i=B(X_i,R)$, see \eqref{eq_stokeslet},  and where $u_{f,N}$ is the solution of the Stokes equation 
\begin{equation}\label{eq:u_f}
\left\{
\begin{array}{rcl}
-\Delta u_{f,N}+\nabla q_{f,N} &=& f (1-1_{\cup B_i}),  \text{ on }(0,+\infty)\times \R^2 \\
\div(u_{f,N})&=&0,  \text{ on }(0,+\infty)\times \R^2 \\
u_{f,N}\vert_{x_1=0}&=&0 
\end{array}
\right.
\end{equation}
We first aim to estimate the error between $u^N$ and 
$u_{\text{app}}$.
\subsection{From $u^N$ to $u_{\text{app}}$}
\begin{prop}
Let $K \Subset \overline{\Omega^0}$. We have for any $q<3/2$
$$
\|u^N-u_{\text{app}}\|_{L^q(K)}\leq C_{K} \phi
$$
\end{prop}
\begin{proof}
We set $v=u^N-u_{\text{app}}$, it satisfies

\begin{equation}
\left\{
\begin{array}{rcl}
-\Delta v+\nabla q &=& 0,  \text{ on }\Omega^0 \setminus \overline{\bigcup B_i}\\
\div(v)&=&0,  \text{ on }\Omega^0 \setminus \overline{\bigcup B_i}\\
Dv &=&- D(u_{\text{app}}), \text { on } B_i\\
v\vert_{x_1=0}&=&0 
\end{array}
\right.
\end{equation}
$$
\int_{B_i} \sigma(v,q) n = 
\int_{B_i} [\sigma(v,q) n] \times(x-X_i)=0 $$

In order to estimate the $L^p_{\text{loc}}$ norm of $v$, we use a duality argument by observing that for any $K\Subset \overline{\Omega^0}$  and $\psi \in C^\infty_c(K)$
$$
\int_{K} v \cdot \psi = -2 \int_{\Omega^0} D (v): D u_\psi
$$
with $u_\psi \in W^{2,q'}(\Omega^0) \cap W^{1,q'}_0(\Omega^0)$ the unique solution to the Stokes equation $-\Delta u_\psi+\nabla q_\psi=\psi $, $\div u_\psi=0$ on $\Omega^0$ with vanishing Dirichlet boundary condition on $x_1=0$ satisfying 
\begin{equation}\label{reg_u_psi}
\|u_\psi \|_{W^{2,q'}(\Omega^0)} \leq C_K \|\psi\|_{L^{q'}(K)}
\end{equation}
Hence we get by an integration by parts

\begin{align*}
2\int_{\Omega^0} D(v): D u_\psi &= 2\int_{\Omega^0\setminus \bigcup_i {B}_i} D(v): D u_\psi - \underset{i}{\sum} \int_{B_i} D(u_{\text{app}}): D(u_\psi)\\
&=\underset{i}{\sum} \int_{\partial B_i} \sigma(v,q)n \cdot u_\psi - \underset{i}{\sum} \int_{B_i} D(u_{\text{app}}): D(u_\psi)\\
&=\underset{i}{\sum} \int_{\partial B_i} \sigma(v,q)n \cdot (u_\psi - \dashint u_\psi) - \underset{i}{\sum} \int_{B_i} D(u_{\text{app}}): D(u_\psi)\\
&= \underset{i}{\sum} \int_{A_i} Dv:D \overline{u_{\psi,i}} - \underset{i}{\sum} \int_{B_i} D(u_{\text{app}}): D(u_\psi)\\
\end{align*}
where, using Lemma \ref{lemme:extension} ,$\overline{u_{\psi,i}} \in W^{1,q'}_0(A_i)$ is a divergence free lifting of $ \overline{u_{\psi,i}}=u_\psi - \dashint_{B_i} u_\psi$ on $B_i$ where $A_i =B(X_i,\theta r)\setminus \overline{B(X_i,r)}$ such that 
\begin{equation}\label{eq:bound_u_psi_i} 
\|\na \overline{u_{\psi,i}}\|_{L^{q'}(A_i)}\leq C_\theta \|\nabla u_\psi \|_{L^{q'}(B_i)} 
\end{equation}
Note that the sets $A_i$ are disjoint thanks to \eqref{hyp_dmin}. Using the well-known estimate 
$$\|Dv\|_{L^2(\Omega^0)} \le C\|Dv\|_{L^2(\bigcup_i B_i)} = \|Du_{\text{app}}\|_{L^2(\bigcup_i B_i)}$$
we get using H\"older inequality, estimate \eqref{reg_u_psi} together with Sobolev embedding
$$
\left |\int_K v \cdot \psi \right| \leq C_K \| Du_{\text{app}}  \|_{L^2(\bigcup_i B_i)} \phi^{1/2} \|\psi\|_{L^{q'}(K)} \leq C_K  \phi\| Du_{\text{app}} \|_{L^\infty(K)} \|\psi\|_{L^{q'}(K)} 
$$
We conclude by observing that $$Du_{\text{app}}= - \underset{j \neq i}{\sum} D\mathcal{U}^{st}_R[F] (x,X_j)- D u_{f,N}. $$
For the second term at the right-hand side, we use standard Sobolev embedding and Stokes estimates. For any $p > 3$:
\begin{align*}
\|Du_{f,N}\|_{L^\infty(K)} \le C \|Du_{f,N}\|_{W^{1,p}(K)} \le C' \|f (1-1_{\cup B_i})\|_{L^p(\R^3)} \le C' (\|f\|_{L^\infty} + \|f\|_{L^1})
\end{align*}
For the first term at the right-hand side we use Proposition \ref{prop:stokeslet} which yields 
\begin{align*}
\| \underset{j \neq i}{\sum} D\mathcal{U}^{st}_R[F] (\cdot,X_j)\|_{L^\infty(B_i)}& \leq \frac{1}{N} \underset{j \neq i}{\sum} \frac{C_\theta }{|X_i-X_j|^2}  \\
& \leq  \frac{C}{N}\underset{j \neq i}{\sum}\dashint_{B(X_j,d_{\min}/4)} \frac{1}{|X_i-y|^2}dy \\
&\leq  \int_{{}^c B(X_i,d_{\min}/4)} \frac{f^N(y)dy}{|X_i-y|^2}\\
&\leq C(\|f^N\|_\infty+\|f^N\|_1)\leq C
\end{align*}
with $f^N=\frac{1}{N} \underset{i}{\sum} \frac{1_{B(X_i,d_{\min}/4)}}{|B(X_i,d_{\min}/4|}$ such that $\|f^N\|_1=1$ and $\|f^N\|_\infty \leq \frac{C}{N d_{\min}^3}$ which is bounded thanks to \eqref{hyp_dmin}.
\begin{rem}\label{discussion_d_min}
More generally, any assumption ensuring 
$$
\frac{1}{N} \underset{j \neq i}{\sum} \frac{1 }{|X_i-X_j|^2} \leq C
$$
is sufficient. This can be ensured for instance by assuming the following bound on the infinite Wasserstein distance
$$
W_\infty(\rho^N,\rho) \leq CN^{-1/3}
$$
and the constraint 
$$ \frac{1}{\sqrt{N}} \lesssim d_{min}$$
Indeed one has
\begin{align*}
\frac{1}{N}\underset{j \neq i}{\sum} \frac{1}{|X_i-X_j|^2}&= \int_{Tx\neq X_i}\frac{1}{|Tx-X_i|^2} \rho(x) dx   
\end{align*}
with $ T$ the optimal transport plan for the $W_\infty:=W_\infty(\rho^N,\rho)$ Wasserstein distance, $\rho^N=T\# \rho$. We split the integral into two parts 
$$E_1=\{ x \in \supp \rho, Tx \neq X_i, |x-X_i| \leq 2W_\infty \}$$
$$E_2 = \{ x \in \supp \rho, Tx \neq X_i, |x-X_i|>2W_\infty \}$$
Hence on $E_1$ one has 
$$
\int_{E_1}\frac{1}{|Tx-X_i|^2} \rho(x) dx  \leq  C\frac{W_\infty^3}{d_{min}^2}
$$
On $E_2$ we have $|Tx-X_i|\geq |x-X_i|-|Tx-x|\geq |x-X_i|/2  $ since $ |Tx-x| \leq W_\infty$ hence
$$
\int_{E_2}\frac{1}{|Tx-X_i|^2} \rho(x) dx  \leq \int \frac{1}{|x-X_i|^2} \rho(x) dx \leq C(\|\rho\|_\infty+\|\rho\|_1)
$$
which shows that 
$$
\frac{1}{N}\underset{j \neq i}{\sum} \frac{1}{|X_i-X_j|^2} \leq C\left(1+\frac{W_\infty^3}{d_{\min}^2}\right)
$$
see also \cite[Lemma 2.3]{HoferSchubert23} for more details.
\end{rem}
\end{proof}
\subsection{From $u_{\text{app}}$ to $v^N$}
We introduce the intermediate velocity 
\begin{equation} \label{eq:Stokes_v^N}
\begin{aligned}
-\Delta v^N + \na q^N   & =-\frac{1}{N} \underset{i}{\sum}\delta_{X_i} e+f(1-1_{\bigcup_i B_i}) , \quad \text{ in } \: \Omega^0 \\
\div v^N & = 0,  \quad \text{ in } \: \Omega^0 \\
 v^N \vert_{x_1=0} &  = 0 
\end{aligned}
\end{equation}
which writes 
$$
v^N=-\frac{1}{N}\underset{i}{\sum}G(\cdot,X_i) e+u_{f,N}
$$
We use again the duality argument to get
\begin{prop}
Let $K \Subset \overline{\Omega^0}$. We have for any $q<3/2$
$$\| u_{\text{app}}- v^N\|_{L^q(K)} \leq  C_K R$$
\end{prop}
\begin{proof}
Given $K \Subset \overline{\Omega^0}$ and $\psi \in C^\infty_c(K)$, we set again $u_\psi \in W^{2,q'}(\Omega^0) \cap W^{1,q'}_0(\Omega^0)$ the unique solution to the Stokes equation $-\Delta u_\psi+\nabla q_\psi=\psi $, $\div u_\psi=0$ on $\Omega^0$ with vanishing Dirichlet boundary condition on $x_1=0$ satisfying 
\begin{equation}\label{eq:estimation_dualite}
\|u_\psi \|_{W^{2,q'}(\Omega^0)} \leq C_K \|\psi\|_{L^{q'}(K)}
\end{equation}
we have using \eqref{eq:Stokes_v^N} and the same notation for $u_\psi$, $\overline{u_{\psi,i}}$ as in the previous proof
\begin{align*}
\int_K(u_{\text{app}}-v^N)\cdot \psi &=\int_{\Omega^0} D(u_{\text{app}}-v^N): Du_{\psi}\\
&=\underset{i}{\sum}  \int_{\Omega^0\setminus B_i }  D(\mathcal{U}^{st}_R[F](\cdot,X_i)) : Du_{\psi} + \underset{i}{\sum}\frac{1}{N} e \cdot u_\psi(X_i)\\
&= \underset{i}{\sum} \int_{\partial B_i}\sigma \left( \mathcal{U}^{st}_R[F](\cdot,X_i),\mathcal{P}^{st}_{r}[F](\cdot,y_0)\right) n \cdot u_{\psi}  +\underset{i}{\sum} \frac{1}{N} e \cdot u_\psi(X_i)\\
&= \underset{i}{\sum} \int_{\partial B_i}\sigma \left( \mathcal{U}^{st}_R[F](\cdot,X_i), \mathcal{P}^{st}_{r}[F](\cdot,y_0)\right) n \cdot \left(u_{\psi}- \dashint_{B_i} u_\psi\right) \\
&+ \underset{i}{\sum}\frac{1}{N} e \cdot \left(u_\psi(X_i)-\dashint_{B_i} u_\psi \right)\\
&=\underset{i}{\sum} \int_{A_i} D \mathcal{U}^{st}_R[F](\cdot,X_i): D\overline{u_{\psi,i}} + \underset{i}{\sum}\frac{1}{N} e \cdot \left(u_\psi(X_i)-\dashint_{B_i} u_\psi \right)\\
& \leq C \underset{i}{\sum} \left \| D \mathcal{U}^{st}_R[F](\cdot,X_i)  \right\|_{L^2(A_i)} |B_i|^{1/2} \|\nabla u_\psi\|_{L^\infty(K)}+ CR \|\nabla u_\psi\|_{L^\infty(K)}\\
& \leq CR \|\psi\|_{L^{q'}(K)} 
\end{align*}
where we used for the first term in the right hand side H\"older inequality, estimate \eqref{eq:bound_u_psi_i} together with the bound $$\|\nabla u_\psi\|_{L^\infty(K)} \le C \|\nabla u_\psi\|_{W^{1,q'}(K)} \le C' \|\psi\|_{L^{q'}(K)}$$ 
and the energy bound \eqref{eq:energy} for $|F|=1/N$. 
\end{proof}

\subsection{From $v^N$ to $u^\varepsilon$}
Keeping in mind that we can restrict to $\eps = 0$, the last step in the proof of Theorem \ref{thm1} is to establish: 
\begin{prop}
Let $K \Subset \overline{\Omega^0}$. We have for any $q<3/2$
$$
\|v^N-u^0\|_{L^q(K)} \leq C_{K} \big(  \|{\rho}^N-\rho\|_{(W^{2,q'}(\Omega^0))^*} + \phi \big) 
$$
where $u^0$ is the solution of 
\begin{equation*} 
\begin{aligned}
-\Delta u^0 + \na p^0 & = f -  \rho e, \quad \text{ in } \: \Omega^0 \\
\div u^0 & = 0,  \quad \text{ in } \: \Omega^0 \\
 u^0 \vert_{\pa \Omega^\eps} &  = 0 
\end{aligned}
\end{equation*}
\end{prop}
\begin{proof}
Using again the duality argument, we have for any $\psi \in C^\infty_c(K)$
\begin{align*}
\int_K (v^N-u^\varepsilon)\cdot \psi & = \int_{\Omega^0} \nabla (v^N-u^\varepsilon): \nabla u_\psi  \\
&= \int_{\Omega^0} ({\rho}^N-\rho) g \cdot u_\psi   - \int_{\bigcup_i B_i} f \cdot u_\psi  \\
&\leq g \|{\rho}^N-\rho\|_{(W^{2,q'}(\Omega^0))^*} \| u_\psi \|_{W^{2,q'}(\Omega^0)} + \sum_i |B_i| \|f\|_{L^\infty(K)}  \|\| u_\psi \|_{L^\infty(K)}\\
&\leq C_{K} \big(  \|{\rho}^N-\rho\|_{(W^{2,q'}(\Omega^0))^*} + \phi  \big) \|\psi\|_{L^{q'}(K)}
\end{align*}
\end{proof}
\section{The approximate continuous model} \label{sec_wall_law}The purpose of this section is to analyse system \eqref{limit_equation}, and to prove Theorem \ref{thm2} and Corollary \ref{coro1}. We assume here that $f$ and $\rho$ are smooth and compactly supported in $\overline{\Omega^0}$.

\subsection{Asymptotic expansion}
We will construct an approximation of the solution $(u^\eps,p^\eps)$ of \eqref{limit_equation}, of the following form:
\begin{equation*}
\left\{
\begin{aligned}
(u^\eps_{app},p^\eps_{app})(x_1,x_2,x_3) & \approx \sum_{i=0}^m \eps^i (u^i, p^i)(x_1,x_2,x_3), \quad x = (x_1,x_2,x_3) \in \Omega^0 \\
(u^\eps_{app}, p^\eps_{app}(x_1,x_2,x_3) & \approx \sum_{i=0}^m \eps^i (U^i, P^i)(x_1/\eps,x_2,x_3), \quad x = (x_1,x_2,x_3) \in \Omega^\eps\setminus\Omega^0
\end{aligned}
\right.
\end{equation*}
Note that this approximation is made of two parts: an interior one, in $\Omega^0$, with a regular expansion in $\eps$, and a boundary layer part, localized in the depletion layer $\Omega^\eps \setminus \Omega^0$. The boundary layer expansion involves boundary layer profiles $U^i = U^i(s,x_2,x_3)$, with the variable $s \in ]-1,0[$ that stands for 
$x_1/\eps$. It is also convenient to set:
$$ (u^i, p^i)= 0, \quad (U^i, P^i) = 0 \quad \text{ for } \: i < 0.$$
Plugging the interior expansion in the Stokes equation, we find: for all $i \ge 0$, in $\Omega^0$
\begin{equation} \label{stokesint}
\begin{aligned}
    -\Delta u^i + \na p^i & = \delta_{0i} \big(f -  \rho e \big),  \\
    \div u^i & = 0. 
 \end{aligned}
 \end{equation}
Plugging the boundary layer expansion in the Stokes equation, we find: for all $i \ge 0$, for all $(s,x_2,x_3) \in (-1,0) \times \R^2$:
\begin{align}
    -\pa^2_s U^i_1 - (\pa^2_2 + \pa^2_3) U^{i-2}_1 + \pa_s P^{i-1} & = 0, \label{stokesx} \\
-\pa^2_s U^i_2 - (\pa^2_2 + \pa^2_3) U^{i-2}_2 + \pa_2 P^{i-2} & = 0, \label{stokesy}\\
-\pa^2_s U^i_3 - (\pa^2_2 + \pa^2_3) U^{i-2}_3 + \pa_3 P^{i-2} & = 0, \label{stokesz} \\
\pa_s U^i_1 + \pa_2 U^{i-1}_2 + \pa_3 U^{i-1}_3 & = 0 \label{stokesdiv}
\end{align}
The Dirichlet condition $u^\eps_{app}\vert_{\pa \Omega^\eps} \approx 0$ further yields: for all $i \ge 0$, 
\begin{equation} \label{dirichlet}
    U^i\vert_{s = -1} = 0. 
    \end{equation}
Eventually, conditions 
$$ [u^\eps_{app}]\vert_{\pa \Omega^0} \approx 0, \quad [\sigma(u^\eps_{app},p^\eps_{app})n]\vert_{\pa \Omega^0} \approx 0 $$
which reflect continuity of the velocity field and the stress tensor at the interface $\pa \Omega^0$ give for all $i \ge 0$:
\begin{align}
U^i\vert_{s=0} & = u^i\vert_{x_1=0} \label{continuity_vel} \\
\pa_s U^i \vert_{s=0} - P^{i-1}\vert_{s=0} (1,0,0)^t & =  \pa_1 u^{i-1}\vert_{x_1=0} - p^{i-1}\vert_{x_1=0} (1,0,0)^t.  \label{continuity_stress}
\end{align}

\paragraph{Computation of the first terms.} We compute $U^0, u^0,p^0$. First, from \eqref{stokesdiv} and \eqref{dirichlet}, we get $\pa_s U^0_1 = 0$ and $U^0_1\vert_{s=-1} = 0$, which implies $U^0_1 = 0$. Then, from \eqref{stokesy}, \eqref{dirichlet} and \eqref{continuity_stress}, we find 
$$ \pa^2_s U^0_2 = 0, \quad U^0_2\vert_{s=-1} = 0, \quad \pa_s U^0_2\vert_{s=0} = 0 $$
which leads to $U^0_2 = 0$. Similarly, $U^0_3 = 0$. Hence, $U^0 = 0$.  Next, we consider \eqref{stokesint} and \eqref{continuity_vel} at rank $i=0$. They imply that $(u^0,p^0)$ satisfies \eqref{limit_equation}. In particular, as $f$ belongs to $L^{6/5}(\Omega_0) \cap H^\infty(\Omega^0)$, $(\na u^0,p^0)$ belongs to $H^\infty(\Omega^0)$ for all $m$.

\paragraph{Computation of next order terms.} We now assume that $i \ge 1$ and that profiles $U^k,P^{k-1}, u^k, p^k$ are known for all $k \le i-1$, with: 
$$ U^k, P^{k-1} \in H^\infty(]-1,0[ \times \R^2), \quad \na u^k, p^k \in H^\infty(\Omega^0)$$
We now show how to construct $U^i,P^{i-1}, u^i, p^i$. Considering \eqref{stokesdiv} and \eqref{dirichlet} yields 
$$ U^i_1(s,\cdot) = - \int_{-1}^s (\pa_2 U^{i-1}_2 + \pa_3 U^{i-1}_3)(t,\cdot) dt \in H^\infty(]-1,0[ \times \R^2).$$
Next, we consider \eqref{stokesx} and \eqref{continuity_stress} which allow the calculation of $P^{i-1} \in H^\infty(]-1,0[ \times \R^2)$:
\begin{align*}
P^{i-1}(s,\cdot) & = \pa_s U^i_1\vert_{s=0} - \pa_1 u^{i-1}_1\vert_{x_1=0} + p^{i-1}\vert_{x_1=0}  \\
& + \int_0^s \big( \pa^2_s U^i_1 + (\pa^2_2 + \pa^2_3) U^{i-2}_1 \big)(t,\cdot) dt
\end{align*}
Next, we consider \eqref{stokesy}-\eqref{dirichlet}-\eqref{continuity_stress}: these relations imply
\begin{align*}
U^i_2(s,\cdot) & = - \int_{-1}^s \int_{0}^t \Big( (\pa^2_2 + \pa^2_3) U^{i-2}_2 + \pa_2 P^{i-2} \Big)(t',\cdot) dt' dt \\
& + (s+1) \pa_1 u^{i-1}_2\vert_{x_1=0} 
\end{align*}
A similar expression holds for $U^i_3$. Both $U^i_2, U^i_3 \in  H^\infty(]-1,0[ \times \R^2)$. Eventually, we see by \eqref{stokesint} and \eqref{continuity_vel} that $(u^i,p^i)$ solve a homogeneous Stokes equation with the inhomogeneous Dirichlet condition
$$ u^i\vert_{x_1=0} = U^i\vert_{x_1=0}.$$
By standard regularity results for the Stokes equation, {\it cf.} 
\cite[Theorem IV.3.2 and IV.3.3]{Galdi}, $\na u^i, p^i \in H^\infty(\Omega^0)$. 

\paragraph{Equations for $u^1$ and $u^2$.}
Anticipating the discussion in section \ref{sec_final_comments}, it is worth  specifying the systems satisfied by $u^1,p^1$ and $u^2,p^2$. From above equations, we deduce  
$$ U^1_1 = 0, \quad P^0 = -\pa_1 u^0_1\vert_{x_1=0} + p^0\vert_{x_1=0}, \quad (U^1_2, U^1_3) = (s+1)  (\pa_1 u^0_2,\pa_1 u^0_3)\vert_{x_1=0}. $$
Hence, 
\begin{equation} \label{systemu1}
    \begin{aligned}
          -\Delta u^1 + \na p^1 & = 0,  \\
    \div u^1 & = 0, \\
    u^1\vert_{x_1=0} & = (0,\pa_1 u^0_2, \pa_1 u^0_3)\vert_{x_1=0}
    \end{aligned}
\end{equation}
Then, 
\begin{align*}
    U^2_1(s,\cdot)  &= - \int_{-1}^s (\pa_2 U^{1}_2 + \pa_3 U^{1}_3)(t,\cdot) dt \\
    & = - \int_{-1}^s (t+1)(\pa_2 \pa_1 u^0_2 + \pa_3 \pa_1 u^0_3)\vert_{x_1=0}dt  = \frac{(s+1)^2}{2} \pa^2_1 u^0_1\vert_{x_1=0}
\end{align*} 
while 
\begin{align*}
(U^2_2, U^2_3)&= (s+1)   (\pa_1 u^1_2, \pa_1 u^1_3)\vert_{x_1=0} + \frac{1}{2} (s-1) (s+1) (\pa_2 P^0, \pa_3 P^0) \\
& = (s+1)  ( \pa_1 u^1_2,  \pa_1 u^1_3)\vert_{x_1=0}\\
& + \frac{1}{2} (s-1) (s+1) \big( -\pa_1\pa_2 u^0_1 + \pa_2 p^0,  -\pa_1\pa_3 u^0_1 + \pa_3 p^0)\vert_{x_1=0}
\end{align*}
We recover the system 
\begin{equation} \label{systemu2}
    \begin{aligned}
          -\Delta u^2 + \na p^2 & = 0,  \\
    \div u^2 & = 0, \\
    u^2_1\vert_{x_1=0} & = \frac{1}{2}\pa^2_1 u^0_1\vert_{x_1=0} \\
    u^2_2\vert_{x_1=0} & = \pa_1 u^1_2\vert_{x_1=0} + \frac{1}{2}( \pa_1\pa_2 u^0_1 - \pa_2 p^0)\vert_{x_1=0}\\
   u^2_3\vert_{x_1=0} & =  \pa_1 u^1_3\vert_{x_1=0} + \frac{1}{2}( \pa_1\pa_3 u^0_1 - \pa_3 p^0)\vert_{x_1=0} 
    \end{aligned}
\end{equation}

\subsection{Proof of Theorem \ref{thm2}}
We will show here that the expansion constructed in the previous section provides an approximate solution of the true solution $u^\eps$ of \eqref{limit_equation}. Given $m \in \N$, we first introduce $\tilde{u}^{m+2}$ the solution in $\dot{H}^1(\Omega^\eps)$ of the Stokes problem 
\begin{equation*}
\left\{
\begin{aligned}
-\Delta \tilde{u}^{m+2} + \na \tilde{p}^{m+2} & = 0 \quad \text{in} \: \Omega^\eps, \\
\div \tilde{u}^{m+2} & = - \pa_1 u^{m+2}_1 1_{\Omega^0}  \quad \text{in} \: \Omega^\eps, \\
\tilde{u}^{m+2}\vert_{\pa \Omega^\eps} & = 0
\end{aligned}
\right.
\end{equation*}
Such $\tilde{u}^{m+2}$ exists, {\it cf.} \cite[Theorem IV.3.3]{Galdi}, and satisfies 
$$\|\na \tilde{u}^{m+2}\|_{L^2(\Omega^\eps)} \lesssim \|\pa_1 u^{m+2}_1 1_{\Omega^0} \|_{L^2(\Omega^\eps)} \lesssim 1$$
We then define $(u^\eps_{app}, p^{\eps}_{app})$ as follows:  
\begin{equation*}
\begin{aligned}
(u^\eps_{app},p^\eps_{app})(x,y,z) & = \sum_{i=0}^{m+1} \eps^i (u^i, p^i)(x,y,z) + \eps^{m+2} (u^{m+2}_1,0,0,0)(x,y,z) \\
& + \eps^{m+2} (\tilde{u}^{m+2},\tilde{p}^{m+2})(x,y,z),  \qquad  X = (x,y,z) \in \Omega^0 \\
(u^\eps_{app}, p^\eps_{app})(x,y,z) & = \sum_{i=0}^{m+1} \eps^i (U^i, P^i)(x/\eps,y,z) + \eps^{m+2} (U^{m+2}_1,0,0,0)(x/\eps,y,z) \\
& + \eps^{m+2} (\tilde{u}^{m+2},\tilde{p}^{m+2})(x,y,z), \qquad \: X = (x,y,z) \in \Omega^\eps\setminus\Omega^0
\end{aligned}
\end{equation*}
We then introduce $v^\eps = u^\eps - u^\eps_{app}$, $q^\eps = p^\eps - p^\eps_{app}$, solving 
\begin{equation} 
\begin{aligned}
-\Delta v^\eps + \na q^\eps  & =  - \eps^{m+2}  (\Delta u^{m+2}_1,0,0)     \quad \text{ in } \: \Omega^0,   \\
-\Delta v^\eps + \na q^\eps  & =  - \sum_{i=m}^{m+1} \eps^i (\pa^2_2 + \pa^2_3) (\eps U^{i+1}_1,U^i_2,U^i_3) - \sum_{i=m}^{m+1} \eps^i (0, \pa_2 P^i, \pa_3 P^i)             \quad \text{ in } \: D^\eps, \\
\div v^\eps & = 0  \quad \text{ in } \: \Omega^\eps, \\
[v^\eps]\vert_{\pa \Omega^0} & = 0, \\
[\sigma(v^\eps,q^\eps)n]\vert_{\pa \Omega^0} & = - \eps^{m+1}(\eps \pa_1 u^{m+2}_1\vert_{x_1=0},\pa_1 u^{m+1}_2\vert_{x_1=0}, \pa_1 u^{m+1}_3\vert_{x_1=0}), \\  
 v^\eps \vert_{\pa \Omega^\eps} &  = 0 
\end{aligned}
\end{equation}
Performing a simple energy estimate, we find 
\begin{align*}
\int_{\Omega^\eps} |\na v^\eps|^2 = \eps^{m+2} \int_{\Omega^0} \na u_1^{m+2} : \na v_1^\eps - \int_{D^\eps} R^\eps \cdot v^\eps + \int_{\pa \Omega^0} r^\eps \cdot v^\eps 
\end{align*}
where 
\begin{align*}
    R^\eps & =  - \sum_{i=m}^{m+1} \eps^i (\pa^2_2 + \pa^2_3) (\eps U^{i+1}_1,U^i_2,U^i_3) - \sum_{i=m}^{m+1} \eps^i (0, \pa_2 P^i, \pa_3 P^i) = O(\eps^m) \\
    r^\eps & = \eps^{m+1}(\eps \pa_1 u^{m+2}_1\vert_{x_1=0},\pa_1 u^{m+1}_2\vert_{x_1=0}, \pa_1 u^{m+1}_3\vert_{x_1=0}) = O(\eps^{m+1})
\end{align*}
By the Poincaré and trace inequalities 
$$ \| v^\eps \|_{L^2(D^\eps)} \lesssim \eps \|\na v^\eps \|_{L^2(D^\eps)}, \quad  \| v^\eps \|_{L^2(\pa\Omega^0)} \lesssim \sqrt{\eps} \|\na v^\eps \|_{L^2(D^\eps)}    $$
we infer easily
$$ \|\na v^\eps \|_{L^2(\Omega^\eps)} \lesssim \eps^{m+1}$$
Restricting to $\Omega^0$, Theorem \ref{thm2} follows. 

\subsection{Wall law of Navier type : proof of Corollary \ref{coro1}}
We conclude Section \ref{sec_wall_law} by proving Corollary \ref{coro1}. 
We know from \eqref{thm2} that 
$$ u^\eps = u^0 + \eps u^1 + O(\eps^2) \quad \text{ in } \: \dot{H}^1(\Omega^0)$$
Hence, Corollary \ref{coro1} will follow from the estimate 
$$ \|\na (u^S - u^S_{app})\|_{L^2(\Omega^0)} = O(\eps^2), \quad u^S_{app} = u^0 + \eps u^1. $$
The difference $v^{S} =  u^S - u^S_{app}$ satisfies 
 \begin{equation*}
\begin{aligned}
-\Delta v^S + \na q^S & =  0, \quad \text{ in } \: \Omega^0  \\
\div v^S & = 0,  \quad \text{ in } \: \Omega^0 \\
 v^S\vert_{\pa \Omega^0}   & =   \eps (0, \pa_1 v^S_2, \pa_1 v^S_3)\vert_{\pa \Omega^0} + \eps^2 (0, \pa_1 u^1_2, \pa_1 u^1_3)\vert_{\pa \Omega^0}
\end{aligned}
\end{equation*}
Let  $V$ the solution of the Stokes problem: 
\begin{equation*}
\begin{aligned}
-\Delta V + \na Q & =  0, \quad \text{ in } \: \Omega^0  \\
\div V & = 0,  \quad \text{ in } \: \Omega^0 \\
 V\vert_{\pa \Omega^0}   & =  (0, \pa_1 u^1_2, \pa_1 u^1_3)\vert_{\pa \Omega^0}
\end{aligned}
\end{equation*}
Then, $\na V$ belongs to $H^\infty(\Omega^0)$. Moreover, the function $w^S = v^S - \eps^2 V$ satisfies 
 \begin{equation*}
\begin{aligned}
-\Delta w^S + \na r^S & =  0, \quad \text{ in } \: \Omega^0  \\
\div w^S & = 0,  \quad \text{ in } \: \Omega^0 \\
 w^S\vert_{\pa \Omega^0}   & =   \eps (0, \pa_1 w^S_2, \pa_1 w^S_3)\vert_{\pa \Omega^0} + \eps^3 (0, \pa_1 V_2, \pa_1 V_3)\vert_{\pa \Omega^0}
\end{aligned}
\end{equation*}
The standard estimate  
\begin{equation*}
\|\na w^S\|_{L^2(\Omega^0)}^2 + \frac{1}{\eps} \int_{\pa \Omega^0} 
(|w^S_2|^2 + |w^S_3|^2) = \eps^2 \int_{\pa \Omega^0} \big(\pa_1 V_2  \,  w^S_2 \: + \:   \pa_1 V_3 \,  w^S_3\big) 
\end{equation*}
combined with Young's inequality yields:
\begin{equation*}
\|\na w^S\|_{L^2(\Omega^0)}^2 + \frac{1}{2\eps} \int_{\pa \Omega^0} (|w^S_2|^2 + |w^S_3|^2) \le \frac{1}{2}\eps^5 \int_{\pa \Omega^0} \big( |\pa_1 V_2|^2 +  \, |\pa_1 V_3|^2 \big)  
\end{equation*}
It follows that 
$\|\na w^S\|_{L^2(\Omega^0)} = O(\eps^{5/2})$ and finally $\|\na v^S\|_{L^2(\Omega^0)} = O(\eps^2)$. The result follows.

\section{Final comments} \label{sec_final_comments}

\subsection{Apparent slip} \label{subsec_apparent}
By Corollary \ref{coro1}, the approximation of system \eqref{limit_equation} by system \eqref{improved_approximation} is more accurate than the one by \eqref{crude_approx}, where the effect of the depletion layer is not taken into account and  homogeneous boundary conditions are considered. The boundary condition in \eqref{improved_approximation} is called a (Navier) slip boundary condition. To understand this terminology, one can consider the case of a shear flow driven by a constant downward pressure gradient:
$$ - \Delta u + \na p = - e, \quad \div u = 0 \quad \text{ in } (0,1) \times \R$$ 
In the case of Dirichlet conditions $u\vert_{x_1=0} = u\vert_{x_1=1} = 0$, the solution is the usual Poiseuille flow 
$$ u^0(x_1,x_2) = \big(0, \frac{1}{2} x_1 (x_1-1)\big)$$
while in the case of Navier conditions 
$$ u\vert_{x_1=0} = (0, \pa_1 u_2\vert_{x_1=0}), \quad u\vert_{x_1=1} = (0, -\pa_1 u_2\vert_{x_1=1}) $$
we find 
$$ u^S(x_1,x_2) = \big(0,\frac{1}{2} x_1 (x_1-1) - \frac12\big) $$
which has a non-zero downward component at the boundary. This corresponds to a phenomenon of {\em apparent slip}, at the artificial boundary $\pa \Omega^0$ (while real no-slip takes place at the rigid wall $\pa \Omega^\eps$). As mentioned in this introduction, this phenomenon of apparent slip was  noticed in various experiments on sheared suspensions.  It is here validated mathematically by  Corollary \ref{coro1}.  

Let us stress that this improved approximation is established in Section \ref{sec_wall_law} starting from the continuous model \eqref{limit_equation}. Actually, as we are interested in the "real" suspension, governed by \eqref{main_equation}, the point would rather be to show that for some norm 
\begin{equation} \label{improved_real}  
\|u^{N,\eps} - u^S\| \ll \|u^{N,\eps} - u^0\|
\end{equation}
By Theorem \ref{thm1}, for all  $q < 3/2$, , $K \Subset \overline{\Omega_0}$, then
\begin{multline*} \|u^{N,\eps} - u^0\|_{L^q(K)} \ge \|u^\eps - u^0\|_{L^q(K)} -  \|u^{N,\eps} - u^\eps\|_{L^q(K)}\\ \ge \|u^\eps - u^0\|_{L^q(K)} - C (R+\phi + \|\rho^N-\rho\|_{(W^{2,q'}(\Omega^0))^*}) 
\end{multline*}
If the solution $u^1$ of \eqref{systemu1} is non-identically zero on $K$, then by Theorem \ref{thm2}: $\|u^\eps - u^0\|_{L^q(K)} \approx \eps$, so that 
$$ \|u^{N,\eps} - u^0\|_{L^q(K)} \ge C' \eps - C (R+\phi + \|\rho^N-\rho\|_{(W^{2,q'}(\Omega^0))^*}) 
$$ 
On the other hand, combining Theorems \ref{thm1} and  \ref{thm2}, we also have 
$$ \|u^{N,\eps} - u^S\|_{L^q(K)} \le \|u^\eps - u^S\|_{L^q(K)} + \|u^{N,\eps} - u^\eps\|_{L^q(K)} \lesssim \eps^2 + C (R+\phi + \|\rho^N-\rho\|_{(W^{2,q'}(\Omega^0))^*}) $$
This ensures that condition \eqref{improved_real} is satisfied  as soon as $1 \gg \eps \gg R+ \phi + \|\rho^N-\rho\|_{(W^{2,q'}(\Omega^0))^*}$. We believe that the restriction on $\phi$ could be relaxed: proceeding as in \cite{Hofer&Schubert}, we could refine the continuous approximation \eqref{limit_equation}, replacing the constant viscosity by the effective Einstein's viscosity, and go from a $O(\phi)$ to a $O(\phi^2)$ error term. Hence, presumably, the most stringent condition comes from the term $\|\rho^N-\rho\|_{(W^{2,q'}(\Omega^0))^*}$, which is known to be $\gtrsim N^{-1/3}$. Note that this is not a constraint that is specific to this norm, and could be relaxed with another measurement of  the gap between $\rho^N$ and $\rho$. Indeed, except for very peculiar point distributions, integral quantities of the form 
$$ \int f(x) (d \rho^N(x) - \rho(x)) dx = \frac{1}{N} \sum f(X_i) - \int f(x) \rho(x) dx  $$
where $f$ belongs to $C^\infty_c(\Omega^0)$, is $\gtrsim N^{-1/3}$. This can be deduced from classical results for 1D Riemann sums, see \cite{Chui,Tasaki}. 
It follows that our justification of apparent slip requires at least 
$$ \eps \gg N^{-1/3}$$ 
The  case $\eps \lesssim N^{-1/3}$ remains an interesting open problem, where use of the intermediate system \eqref{limit_equation} is probably not relevant. 

\subsection{Intrinsic convection} \label{subsecintrinsic}
In all our analysis, it is implicit that the first terms $u^0$ and  $u^1$  in our expansion of $u^\eps$   are non-zero. This is very natural in view of our assumptions, where the limit density of the particles distribution $\rho$ is {\em inhomogeneous} (and even compactly supported). This inhomogeneity drives  a non-zero global flow with field $u^0$. Such setting excludes the case of a homogeneous suspension, meaning with constant density over 
$\Omega^0$ (and no driving force). Nevertheless, we can still in this degenerate case use the continuous model \eqref{limit_equation}, with $\rho=1$ and $f=0$, and perform at least formally  the asymptotic expansion seen in the proof of Theorem \ref{thm2} (an analogue of this theorem for constant $\rho$ would require to change the functional spaces, as there is no decay anymore at infinity).

Such an approach for homogeneous suspensions can be found (although sketchy) in \cite{BFAH96}, for a 2d channel of the form $(-\eps,1) \times \R$. See also \cite{BFBA98}. This can be straightforwardly adapted to our domain $\Omega^\eps$. First, when $\rho=1$ and $f=0$, the solution of \eqref{crude_approx} is given by 
$$ u^0(x) = 0, \quad p^0(x) = -x_3$$
It follows from \eqref{systemu1} that $u^1=0, \: p^1 = 0$. 
Going one step further and considering \eqref{systemu2}, we find that $u^2$ satisfies the boundary condition 
$$u^2\vert_{x_1=0}= (0,0,\frac{1}{2}) $$
This gives 
$$u^\eps\vert_{x_1=0} \approx (0,0, \frac{1}{2} \eps^2). $$
 Back to dimensional variables, we find 
 $$u^\eps_3\vert_{x_1=0} \approx \frac{\eps^2 mgN}{2 \mu L^3} = \frac{g}{2\mu} \eps^2 (\rho_s-\rho_f) \phi $$
 using that the mass of a spherical particle (rectified by buoyancy) is  $m = \frac{4}{3} \pi R^3 (\rho_p - \rho_f)$ with $\rho_p$ and $\rho_f$ the particle and fluid masses per unit, while the solid volume fraction is $\phi = \frac{4}{3}\pi  R^3/L^3 N$. Introducing the settling speed of a single spherical particle 
 $$ V_0 = 2R^2 (\rho_p - \rho_f)g/9\mu$$
 this can be rewritten 
 $$u^\eps_3\vert_{x_1=0} \approx \frac{\eps^2}{R^2} \frac{9}{4} V_0 \phi. $$
which agrees with the formula found in \cite{BFAH96} for the special case $\eps = R$ considered there. 
We see that contrary to the case examined in paragraph \ref{subsec_apparent}, the velocity just outside the depletion layer (at $x_1=0$) is pointing upward. For zero mean flux in a channel, this upward velocity near the layers is compensated by an extra downward velocity in the middle of the channel, as discussed in \cite{BFAH96}. This extra flow is tagged as intrinsic convection in the literature. 

Still, we stress that this formal analysis is conducted at the level of the continuous model \eqref{limit_equation}, which as discussed in paragraph 
\ref{subsec_apparent} is at most accurate at order $N^{-1/3}$. It means that to establish rigorously the phenomenon of intrinsic convection through this path, we need the size $\eps^2$ of the extra velocity to  be larger than $N^{-1/3}$. This corresponds to much wider depletion layers than those considered in \cite{BFAH96}. Beyond the question of mathematical justification, this may explain the lack of experimental evidence for this phenomenon.

\appendix
\section{Estimates for the fundamental solution of the Stokes equation on the half space}\label{appendix}
Let us first introduce some notations. We denote by $(\Phi,Q)$ the fundamental solution to the Stokes equation on $\R^3$ given by
$$
\Phi(x)=\frac{1}{8\pi}\left(\frac{\mathbb{I}}{|x|}+\frac{x\otimes x}{|x|^3} \right),\quad Q(x)=\frac{1}{4\pi} \frac{x}{|x|^3}
$$
where $\mathbb{I}$ the identity matrix in $3d$. For any $F=(F_1,F_2,F_3) \in \R^3$, we denote by $F^I=(-F_1,F_2,F_3)$.
We rely on \cite[Section 3.1]{Gimbutas_Greengard_Veerapaneni_2015} for the following result 
\begin{prop}\label{prop_appendix}
Let $y \in (0,+\infty)\times \R^2$  and denote by $G(x,y)$ the unique solution to  
$$-\Delta_x G+\nabla_x p = \delta_{y} \mathbb{I}, \quad  \div_x G= 0, \text{ on } \: (0,+\infty)\times \R^2, \quad G(x,y)\vert_{x_1=0}=0$$ 
\begin{enumerate}
\item We have 
$$G(x,y)=\Phi(x-y)-\Phi(x-y^I)-A(x,y)$$
with 
$$A_{i,j} = (x_1 \partial_{x_i}-\delta_{i1})\left( \frac{1}{4\pi}\frac{(e_j^I)_1}{|x-y^I|}+\frac{y_1}{4\pi}   \frac{(x-y^I)\cdot e_j^I}{|x-y^I|^3} \right), \quad 1 \leq i,j\leq 3$$
where $\{e_1,e_1,e_3\}$ stands for the canonical basis of $\R^3$.
\item There exists a constant $C>0$ such that for all $x,y \in (0,+\infty)\times \R^2 $, $x\neq y$
$$
|G(x,y)| \leq \frac{C}{|x-y|}, \quad | \nabla_1 G(x,y)| + | \nabla_y G(x,y)| \leq \frac{C}{|x-y|^2}
$$
\end{enumerate}
\end{prop}
\begin{proof}
\begin{enumerate}
    \item The explicit formula of $G$ can be found in the nice paper \cite[Formula (21)]{Gimbutas_Greengard_Veerapaneni_2015}. Let us remind for completeness the main elements of its derivation. Let $F\in \R^3$, the idea is that subtracting to $\Phi(x-y)F$ its reflected point force $\Phi(x-y^I)F^I$ yields a velocity $u$
    $$
    u(x)=\Phi(x-y)F-\Phi(x-y^I)F^I$$
which satisfies $-\Delta u +\nabla p=F \delta_y$, $\div u=0$ on $(0,+\infty)\times\R^2$   with $u_3(x)_{|_{x_1=0}}=u_2(x)_{|_{x_1=0}}=0$ and
$$
u_1(x)_{|_{x_1=0}}=- \frac{1}{4\pi}\frac{F_1^I}{|x-y^I|}-\frac{y_1}{4\pi} \frac{F^I \cdot (x-y^I)}{|x-y^I|^3}
$$
At the right-hand side,  one recognizes the harmonic potential due to a simple charge and the one due to a  dipole with orientation $F^I$. 
Such an inhomogenous Dirichlet condition can be corrected by introducing the so called Papkovich-Neuber correction \cite[Definition 2.1]{Gimbutas_Greengard_Veerapaneni_2015}
    $$
    u^C(x)=x_1 \nabla \phi-e_1 \phi(x), \quad p^C(x)=2\partial_{x_1}\phi(x)
    $$
which satisfies $-\Delta u^C+\nabla p^C =0$, $\div u^C=0$ for any harmonic function $\phi$ and $u^C_{|_{x_1=0}}=\phi(x)_{|_{x_1=0}}e_1$. Hence in order to correct the boundary condition of $u$ one needs to add the term $u^C$,   
    with 
    $$
    \phi(\mathbf{x})=\frac{1}{4\pi}\frac{F_1^I}{|x-y^I|}+\frac{y_1}{4\pi} \frac{F^I \cdot (x-y^I)}{|x-y^I|^3}
    $$
    This yields the claimed formula. 
    \item     For the decay, we use the following properties for any $x, y \in (0,+\infty)\times \R^2 $, $x\neq y$ 
    $$|x-y^I|\geq |x-y|, \quad \frac{|y_1|}{|x-y^I|} \leq 1, \quad \frac{|x_1|}{|x-y^I|} \leq 1$$
    together with the following estimates 
$$
|A(x,y)| \leq C \left(\frac{|x_1|}{|x-y^I|}+ 1 \right) \left(\frac{1}{|x-y^I|}+ \frac{|y_1|}{|x-y^I|^2} \right)
$$
$$
|\nabla_x A(x,y)|+|\nabla_y A(x,y)| \leq C \left(\frac{|x_1|}{|x-y^I|}+ 1 \right) \left(\frac{1}{|x-y^I|^2}+ \frac{|y_1|}{|x-y^I|^3} \right)
$$
\end{enumerate}
\end{proof}


\subsection*{Acknowledgment}
The authors would like to thank Matthieu Hillairet and Richard H\"ofer for stimulating discussions on the topic. The work of D. G-V was supported by project Singflows grant ANR-18-CE-40-0027 and project Bourgeons grant ANR-23-CE40-0014-01 of the French National Research Agency (ANR).

\end{document}